\documentclass[draft]{amsart}

\newtheorem{thm}{Theorem}[section]

\newtheorem{lem}[thm]{Lemma}

\theoremstyle{remark}

\theoremstyle{definition}

\numberwithin{equation}{section}
\numberwithin{thm}{section}

\begin{document}

%%%%%%%%%%%%%%%%%%%%%%%%%
% Subject classification
%%%%%%%%%%%%%%%%%%%%%%%%%

% Provide an AMS subject classification with one or two primary classification
% numbers and, optionally, one or more secondary classification numbers.
% Use the following format:  "Primary 42B25. Secondary 42B60, 20E26"


%%%%%%%%%
% Title
%%%%%%%%%

% Title, in lower case, with no explicit linebreaks (\\).  If the title
% is too long to be used as a running head, add a short version of the
% title in brackets, as in \title[shorttitle]{fulltitle}.

\title[fractional integral operators on non-homogeneous spaces]{Boundedness of fractional integral operators on non-homogeneous metric measure spaces}

%%%%%%%%%%%%%%%%%%%%%%%%%%%%%%
% Author names and addresses
%%%%%%%%%%%%%%%%%%%%%%%%%%%%%%

% Provide one separate \author{...} \address{...} \email{....} entry for each
% author, i.e., do not combine multiple authors.  Separate address lines by double
% slashes.  Do not attach footnotes to author  names. (For acknowledgements use
% the "\thanks" construct below.)
%

\author{Rulong Xie and Lisheng Shu}
\address{Rulong Xie, 1. School of Mathematical Sciences, University of
Science and Technology of China, Hefei 230026, China 2. Department
of Mathematics, Chaohu University, Chaohu 238000,
China}\email{xierl@mail.ustc.edu.cn}
\address{ Lisheng Shu, Department of Mathematics, Anhui Normal University,
Wuhu 241000, China} \email{shulsh@mail.ahnu.edu.cn}

%%%%%%%%%%%%%%%%%%%%
% Acknowledgements
%%%%%%%%%%%%%%%%%%%

% Use \thanks for acknowledgements as footnotes to the title page.
% (Note that footnotes inside \author or \title macros are not
% allowed.)
%
% In case of multiple author papers, phrase the acknowledgement to
% say "The first author was supported by ...  The second author was
% supported by ..."

\thanks{Supported by National Natural Science Foundation of China
(No.10371087) and Natural Science Foundation of Education Committee
of Anhui Province (No.KJ2011A138, No.KJ2012B116). }

%%%%%%%%%%%%%
% Abstract
%%%%%%%%%%%%%
%
% Abstracts should not contain macros (so that they can be processed independently
% of the paper.) Avoid displayed math and references in the abstract.

\begin{abstract}
In this paper, the fractional integral operator on non homogeneous
metric measure spaces is introduced, which contains the classical
fractional integral operator, fractional integral with non-doubling
measures and fractional integral with fractional kernel of order
$\alpha$ and regularity $\epsilon$ introduced by Garc\'{i}a-Cuerva
and Gatto as special cases. And the
$(L^{p}(\mu),L^{q}(\mu))$-boundedness for fractional integral
operators on non-homogeneous metric measure spaces is established.
From this, the $(L^{p}(\mu),L^{q}(\mu))$-boundedness for commutators
and multilinear commutators generated by fractional integral
operators with $RBMO(\mu)$ function are further obtained. These
results in this paper include the corresponding results on both the
homogeneous spaces and non-doubling measure spaces.
\end{abstract}

\maketitle

%%%%%%%%%%%%%%%%%%%%%%%%%%%%%%%%%%%%%%%%%%%%%%%%%%%%%%%%%%%%%%%%%%%%%%%%%
% end Topmatter
%%%%%%%%%%%%%%%%%%%%%%%%%%%%%%%%%%%%%%%%%%%%%%%%%%%%%%%%%%%%%%%%%%%%%%%%%

%%%%%%%%%%%%%%%%%%%%%%%%%%%%%%%%%%%%%%%%%%%%%%%%%%%%%%%%%%%%%%%%%%%%%%%%%
% body of paper
%%%%%%%%%%%%%%%%%%%%%%%%%%%%%%%%%%%%%%%%%%%%%%%%%%%%%%%%%%%%%%%%%%%%%%%%%

\section{Introduction}

As we know, the theory on spaces of homogeneous type need to assume
that measure $\mu$ of metric spaces $(X,d,\mu)$ satisfies the
doubling measure condition, i.e. there exists a constant $C>0$ such
that $\mu(B(x,2r))\leq C\mu(B(x,r))$ for all $x\in \rm{supp} \mu$
and $r>0$. Recently, many classical theory have been proved still
valid without the assumption of doubling measure condition, see
[2,4-6,8-10,18,19,21-24] and their references. In case of
non-doubling measures, a Radon measure $\mu$ on $R^{n}$ only need to
satisfy the polynomial growth condition, i.e., for all $x \in R^{n}$
and $r>0$, there exists a constants $C>0$ and $k\in(0, n]$ such
that,
\begin{equation}\
\mu(B(x, r)) \leq C_{0}r^{k},
\end{equation}
where $B(x, r)= \{y \in  R^{n}: |y-x| < r\}$. The analysis on
non-doubling measures has important applications in solving the
long-standing open Painlev$\acute{e}$'s problem (see [21]).

   However, as pointed out by Hyt\"{o}nen in [12], the measure
satisfying (1.1) do not include the doubling measure as special
cases. For this reason, a kind of metric measure spaces $(X,d,\mu)$
satisfying geometrically doubling and the upper doubling measure
condition (see Definition 1.1 and 1.2) is introduced by Hyt\"{o}nen
in [12], which is called non-homogeneous metric measure spaces. The
advantage of this kind of spaces is that it includes both the
homogeneous spaces and metric spaces with polynomial growth measures
as special cases. From then on, some results paralled to homogeneous
spaces and non-doubling measure spaces are obtained (see [1,3,11-17]
and the references therein). For example, Hyt\"{o}nen et al. in [14]
and Bui and Duong in [1] independently introduced the atomic Hardy
space $H^{1}(\mu)$ and obtained that the dual space of $H^{1}(\mu)$
is $RBMO(\mu)$. Hyt\"{o}nen and Martikainen [13] established the
$T_{b}$ theorem in this surroundings. In [1], Bui and Duong also
obtained that Calder\'{o}n-Zygmund operator and commutators of
Calder\'{o}n-Zygmund operators with $RBMO(\mu)$ function are bounded
in $L^{p}(\mu)$ for $1<p<\infty$. Later, Lin and Yang [16]
introduced the space $RBLO(\mu)$ and proved the maximal
Calder\'{o}n-Zygmund operators is bounded from $L^{\infty}(\mu)$
into $RBLO(\mu)$. Recently, some equivalent characterizations are
established by Liu, Yang Da. and Yang Do. in [17] for the
boundedness of Carder¡äon-Zygmund operators on $L^{p}(\mu)$ for
$1<p<\infty$. The boundedness and weak type endpoint estimate of
multilinear commutators of Calder\'{o}n-Zygmund operators on
non-homogeneous metric spaces is established by Fu, Yang and Yuan in
[3]. And weighted estimate for multilinear Calder\'{o}n-Zygmund
operators on non-homogeneous metric spaces is also obtained by Hu,
Meng and Yang in [11].

The purpose of this paper is to establish the theory of fractional
integral operators on non-homogeneous metric measure spaces. At
first, fractional integral operators on non-homogeneous metric
measure spaces is introduced. This kind of fractional integral
operators contains the classical fractional integral operator,
fractional integral with non-doubling measures and fractional
integral with fractional kernel of order $\alpha$ and regularity
$\epsilon$ introduced by Garc\'{i}a-Cuerva and Gatto in [4] as
special cases. The $(L^{p}(\mu),L^{q}(\mu))$-boundedness for
fractional integral operators on non-homogeneous metric measure
spaces is obtained. From this result, the boundedness of commutators
and multilinear commutators generated by fractional integral
operators with $RBMO(\mu)$ function are also established. The
results in this paper include the corresponding results on both the
homogeneous spaces and non-doubling measure spaces.

For the sake of the reader's convenience, let us give some
references about previous results of fractional integral operators
which are closely related to results in this paper. Classical
fractional integral theory can be founded in [20]. In the
circumstance of non-doubling measures,
$(L^{p}(\mu),L^{q}(\mu))$-boundedness for fractional integral
operators is established in [4,5]. The boundedness of commutators
and multilinear commutators generated by fractional integral
operators with $RBMO(\mu)$ function were established in [2] and [8]
respectively. In addition, it's worth mentioning that the
Besicovitch covering lemma is only applicable to $R^{n}$, but is not
applicable to non-homogeneous metric measure spaces. Therefore, in
the process of proving Lemma 2.2 in this paper, we will adopt a
covering lemma introduced in [7].

 Before stating the main results, we firstly recall some notations
and definitions.

\noindent{\bf Definition 1.1.}$^{[12]}$\quad A metric spaces $(X,d)$
is called geometrically doubling if there exist some $N_{0}\in
\mathbb{N}$ such that, for any ball $B(x,r)\subset X$, there exists
a finite ball covering $\{B(x_i,r/2)\}_i$ of $B(x,r)$ such that the
cardinality of this covering is at most $N_0$.

\noindent{\bf Definition 1.2.}$^{[12]}$\quad A metric measure space
$(X,d,\mu)$ is said to be upper doubling if $\mu$ is a Borel measure
on $X$ and there exists a dominating function $\lambda : X
\times(0,+\infty) \rightarrow (0,+\infty)$ and a constant
$C_{\lambda} >0$ such that for each $x\in X, r \longmapsto (x,r)$ is
non-decreasing, and for all $x\in X, r >0$,
 \begin{equation}
 \mu(B(x, r))\leq \lambda(x, r)\leq C_{\lambda}\lambda(x, r/2)
\end{equation}

\noindent{\bf Remark 1.1.}(i)\quad A space of homogeneous type is a
special case of upper doubling spaces, where one can take the
dominating function $\lambda(x, r)\equiv \mu(B(x,r))$. On the other
hand, a metric space $(X,d,\mu)$ satisfying the polynomial growth
condition (1.1) (in particular, $(X,d,\mu)\equiv(R^n, |\cdot|, \mu)$
with $\mu$ satisfying (1.1) for some $k\in (0, n])$) is also an
upper doubling measure space if we take $\lambda(x, r)\equiv
Cr^{k}$.

(ii) \quad Let$(X,d,\mu)$ be an upper doubling space and $\lambda$
be a dominating function on $X \times(0,+\infty)$ as in Definition
1.2. In [15], it was showed that there exists another dominating
function $\widetilde{\lambda}$ such that for all $x, y \in X$ with
$d(x, y)\leq r$,
\begin{equation}
\widetilde{\lambda}(x, r)\leq \widetilde{C}\widetilde{\lambda}(y,
r).
\end{equation}
 Thus, in this paper, we always suppose that $\lambda$ satisfies
(1.3).

\noindent{\bf Definition 1.3.}\quad
 Let $\alpha,\beta \in (1,+\infty)$. A ball $B\subset X$ is called $(\alpha,
\beta)$-doubling if $\mu(\alpha B)\leq \beta \mu (B)$.

As stated in Lemma 2.3 of [1],  there exist plenty of doubling balls
with small radii and with large radii. In the rest of the paper,
unless $\alpha$ and $\beta$ are specified otherwise, by an
$(\alpha,\beta)$ doubling ball we mean a $(6,\beta_{0})$-doubling
with a fixed number $\beta_0 >\max\{C_{\lambda}^{3log_{2}6},
6^{n}\}$, where $n=log_{2}N_{0}$ be viewed as a geometric dimension
of the spaces.

\noindent{\bf Definition 1.4.} Let $0\leq \beta<n$ and $N_{B,Q}$ be
the smallest integer satisfying $6^{N_{B,Q}}r_{B}\geq r_Q$, then we
set
\begin{equation}
K^{(\beta)}_{B,Q} = 1+\sum_{k=1}^{N_{B,Q}}
\biggl[\frac{\mu(6^{k}B)}{\lambda(x_B,6^{k}r_{B})}\biggr]^{1-\beta/n}.
\end{equation}
When $\beta=0$, then we denote $K^{(0)}_{B,Q}$ by $K_{B,Q}$, which
is firstly defined in [22].

\noindent{\bf Definition 1.5.}$^{[1]}$\quad Let $\rho>1$ be some
fixed constant. A function $b\in L_{loc}^{1}(\mu)$ is said to belong
to $RBMO(\mu)$ if there exists a constant $C
>0$ such that for any ball $B$
\begin{equation}
\frac{1}{\mu(\rho B)}\int_{B}|b(x)-m_{\widetilde{B}}b|d\mu(x)\leq C,
\end{equation}
and for any two doubling balls $B\subset Q$,
 \begin{equation}
  |m_{B}(b)-m_{Q}(b)|\leq CK_{B,Q}.
\end{equation}
$\widetilde{B}$ is the smallest $(\alpha,\beta)$-doubling ball of
the form $6^{k}B$ with $k\in {\mathbb{N}}\bigcup\{0\}$, and
$m_{\widetilde{B}}(b)$ is the mean value of $b$ on $\widetilde{B}$,
namely,
$$m_{\widetilde{B}}(b)=\frac{1}{\mu(\widetilde{B})}\int_{\widetilde{B}}b(x)d\mu(x).$$
The minimal constant $C$ appearing in (1.5) and (1.6) is defined to
be the $RBMO(\mu)$ norm of $b$ and denoted by $||b||_{\ast}$. The
norm $||b||_{\ast}$ is independent of $\rho>1$.

Next, let us introduce fracional integral operator on nonhomogeneous
metric measure spaces.

\noindent{\bf Definition 1.6.}\quad Let $0<\alpha <n$ and
$0<\epsilon\leq 1$.  A function $K_{\alpha}(\cdot,\cdot)\in
L_{loc}^{1}(X\times X\backslash\{(x,y):x=y\})$ is said to be a
fractional kernel of order $\alpha$ and regularity $\epsilon$ if it
satisfies the following two conditions:

 (i)\begin{equation}
 \quad |K_{\alpha}(x,y)|\leq
\frac{C}{\biggl[\lambda(x,d(x,y))\biggr]^{1-\alpha/n}}
 \end{equation}
for all $x\neq y$.

 (ii) There exists $0<\epsilon\leq 1$ such that

 \begin{equation}
|K_{\alpha}(x,y)-K_{\alpha}(x',y)|+|K_{\alpha}(y,x)-K_{\alpha}(y,x')|
 \leq \frac{Cd(x,x')^{\epsilon}}{d(x,y)^{\epsilon}\biggl[\lambda(x,d(x,y))\biggr]^{1-\alpha/n}},
\end{equation}
proved that $Cd(x,x')\leq d(x,y)$.

A  operator $I_{\alpha}$ is called a fractional integral operator on
non-homogeneous metric measure spaces with the above fractional
kernel $K_{\alpha}$ satisfying (1.7) and (1.8) if, for $f\in
L^{\infty}$ functions with compact support and $x\notin supp f$,
 \begin{equation}
I_{\alpha}f(x)=\int_{X}K_{\alpha}(x,y)f(y) d\mu(y).
\end{equation}

\noindent{\bf Remark 1.2.}\quad By taking
$\lambda(x,d(x,y))=Cd(x,y)^{n}$, it is easy to see that Definition
1.6 in this paper contains Definition 3.1 and Definition 4.1
introduced by Garc\'{i}a-Cuerva and Gatto in [4]. Obviously, it also
contains the classical fractional integral operator
$$I_{\alpha}f(x)=\int_{R^{d}}\frac{f(y)}{|x-y|^{n-\alpha}}d\mu(y)$$ as special
case.

\noindent{\bf Definition 1.7.}\quad The commutators $[b,I_{\alpha}]$
generated by fractional integral operator $I_{\alpha}$ with
$RBMO(\mu)$ function $b$ is defined by
\begin{equation}[b,I_{\alpha}](f)(x)=b(x)I_{\alpha}(f)(x)-I_{\alpha}(bf)(x).\end{equation}

\noindent{\bf Definition 1.8.}\quad  The multilinear commutators
$I_{\alpha,\vec{b}}$ is formally defined by
 \begin{equation}I_{\alpha,\vec{b}}f(x)=[b_{k},[b_{k-1},\cdot\cdot\cdot,[b_{1},I_{\alpha}]]]f(x)\end{equation}
where $\vec{b}=(b_{1},b_{2},\cdot\cdot\cdot,b_{k})$, and
$$[b_{1},I_{\alpha}]f(x)=b_{1}(x)I_{\alpha}f(x)-I_{\alpha}(b_{1}f)(x).$$

For $1\leq i \leq k$, we denote by $C_{i}^{k}$ the family of all
finite subsets
$\sigma=\{\sigma(1),\sigma(2),\cdot\cdot\cdot,\sigma(i)\}$ of
$\{1,2,\cdot\cdot\cdot,k\}$ with $i$ different elements. For any
$\sigma\in C_{i}^{k}$, the complementary sequences $\sigma'$ is
given by $\sigma'=\{1,2,\cdot\cdot\cdot,k\}\backslash\sigma$. Let
$\vec{b}=(b_{1},b_{2},\cdot\cdot\cdot,b_{k})$ be a finite family of
locally integrable functions. For all $1\leq i\leq k$ and
$\sigma=\{\sigma(1),\cdot\cdot\cdot,\sigma(i)\}\in C_{i}^{k}$, we
set $\vec{b}_{\sigma}=(b_{\sigma(1)},\cdot\cdot\cdot,b_{\sigma(i)})$
and the product $b_{\sigma}=b_{\sigma(1)}\cdot\cdot\cdot
b_{\sigma(i)}$. With this notation, we write
$$||\vec{b}_{\sigma}||_{\ast}=||b_{\sigma(1)}||_{\ast}\cdot\cdot\cdot ||b_{\sigma(i)}||_{\ast}.$$
In particular, for $i\in \{1,\cdot\cdot\cdot,k\}$ and
$\sigma=\{\sigma(1),\cdot\cdot\cdot,\sigma(i)\}\in C_{i}^{k}$,
$$[b(y)-b(z)]_{\sigma}=[b_{\sigma(1)}(y)-b_{\sigma(1)}(z)]\cdot\cdot\cdot[b_{\sigma(i)}(y)-b_{\sigma(i)}(z)],$$
and
$$[b(y)-m_{\widetilde{B}}(b)]_{\sigma}=[b_{\sigma(1)}(y)-m_{\widetilde{B}}(b_{\sigma(1)})]\cdot\cdot\cdot
[b_{\sigma(i)}(y)-m_{\widetilde{B}}(b_{\sigma(i)})],$$ where $B$ is
any ball in $X$ and $y,z\in X$. For the product of all the
functions, we simply write
$$||\vec{b}||_{\ast}=||b_{1}||_{\ast}\cdot\cdot\cdot ||b_{k}||_{\ast}.$$
With any $\sigma\in C_{i}^{k}$, we set
$$I_{\alpha,\vec{b_{\sigma}}}f(x)=[b_{\sigma(i)},[b_{\sigma(i-1)},\cdot\cdot\cdot,[b_{\sigma(1)},I_{\alpha}]]]f(x).$$
In particular, when $\sigma=\{1,\cdot\cdot\cdot,k\}$, we denote
$I_{\alpha,\vec{b_{\sigma}}}$ simply by $I_{\alpha,\vec{b}}$.

The main results of this paper are stated as follows.

\begin{thm}\label{thm-main1.1}
Let $0<\alpha <n$, $1<p<n/\alpha$ and $1/q=1/p-\alpha/n$.
$I_{\alpha}$ is defined by (1.9). Then there exists a constant $C>0$
such that for all $f\in L^{p}(\mu)$,
\begin{equation}
||I_{\alpha}f||_{L^{q}(\mu)}\leq C ||f||_{L^{p}(\mu)}.
\end{equation}

\end{thm}

\begin{thm}\label{thm-main1.2}
Let $0<\alpha <n$, $1<p<n/\alpha$ and $1/q=1/p-\alpha/n$. Suppose
that $b\in RBMO(\mu)$ and $[b,I_{\alpha}]$ is defined by (1.10).
Then there exists a constant $C>0$ such that for all $f\in
L^{p}(\mu)$,
\begin{equation}
||[b,I_{\alpha}](f)||_{L^{q}(\mu)}\leq C
||b||_{\ast}||f||_{L^{p}(\mu)}.
\end{equation}
\end{thm}

\begin{thm}\label{thm-main1.3}
Let $0<\alpha <n$, $1<p<n/\alpha$ and $1/q=1/p-\alpha/n$. Suppose
that $b_{i}\in RBMO(\mu), i=1,2,\cdots,k$ and $I_{\alpha,\vec{b}}$
is defined by (1.11). Then there exists a constant $C>0$ such that
for all $f\in L^{p}(\mu)$,
\begin{equation}
||I_{\alpha,\vec{b}}(f)||_{L^{q}(\mu)}\leq C
||\vec{b}||_{\ast}||f||_{L^{p}(\mu)}.
\end{equation}
\end{thm}

 Throughout this paper,  $\chi_{E}$
denotes the characteristic function of set $E$. $C$ always denotes a
positive constant independent of the main parameters involved, but
it may be different from line to line. And $p'$ is the conjugate
index of $p$, namely, $\dfrac{1}{p}+\dfrac{1}{p'}=1.$

\section{Proof of Theorem 1.1}

To prove Theorem 1.1, we need to give the following notations and
lemmas.

Let $0\leq \beta <n$ and $f \in L_{loc}^{1}(\mu)$, the sharp maximal
operator is defined by
 \begin{equation}
 \begin{split}
M^{\sharp,(\beta)}f(x)&=\sup _{B\ni
x}\frac{1}{\mu(6B)}\int_{B}|f(y)-m_{\widetilde{B}}(f)|d\mu(y)\\
&\quad
+\sup_{(B,Q)\in\Delta_{x}}\frac{|m_{B}(f)-m_{Q}(f)|}{K^{(\beta)}_{B,Q}},
\end{split}
\end{equation}
where $\Delta_{x}:=\{(B,Q):x\in B\subset Q\ \text{and $B, Q$ are two
doubling balls}\}$ and the non centered doubling maximal operator is
denoted by
$$Nf(x)=\sup_{B\ni x,\ B \
doubling}\frac{1}{\mu(B)}\int_{B}|f(y)|d\mu(y).$$ By the Lebesgue
differentiation theorem, it is easy to see that for any $f\in
L_{loc}^{1}(\mu)$ and almost every $x\in X$,
 \begin{equation}|f(x)|\leq Nf(x).\end{equation}
Moreover, $N$ is of weak type (1,1) and bounded on $L^{p}(\mu),
1<p\leq+\infty$.

From Theorem 4.2 in [1] or Lemma 4 in [2], it is easy to obtain that
\begin{lem}%lemma2.1
Let $0\leq \beta <n$ and $f \in L^{1}_{loc}(\mu)$ with
$\int_{X}f(x)d\mu(x)=0$ if $||\mu||<\infty$. For $1<p<\infty$, if
$\inf(1,Nf)\in L^{p}(\mu)$, then there exists a constant $C>0$ such
that
\begin{equation}
||N(f)||_{L^{p}(\mu)}\leq C||M^{\sharp,(\beta)}(f)||_{L^{p}(\mu)}.
\end{equation}
\end{lem}

\begin{lem}%lemma2.2
Let $\eta\geq 5$, $0\leq \beta <n/p$, $r<p<n/\beta$ and
$1/q=1/p-\beta/n$. Then
\begin{equation}
||M_{r,(\eta)}^{(\beta)}f||_{L^{q}(\mu)}\leq C||f||_{L^{p}(\mu)},
\end{equation} where the non-centered maximal operator
$M_{r,(\eta)}^{(\beta)}$ is defined by
\begin{equation}
 M_{r,(\eta)}^{(\beta)}f(x)=\sup_{B\ni x}\biggl\{\frac{1}{\mu(\eta
B)^{1-\beta r/n}}\int_{B}|f(y)|^{r}d\mu(y)\biggr\}^{1/r}.
\end{equation}
\end{lem}
\noindent{\bf Remark 2.1.}\quad If $\beta=0$, we denote
$M_{r,(\eta)}^{(\beta)}$ simply by $M_{r,(\eta)}$. From [1], we
obtain that if $\eta\geq 5$ and $q>r$, then
\begin{equation}
||M_{r,(\eta)}f||_{L^{q}(\mu)}\leq C||f||_{L^{q}(\mu)}.
\end{equation}

\begin{proof} (of Lemma 2.2) %proof of Lemma2.2
Let us firstly prove the following result.
\begin{equation}
\mu(\{x: M_{r,(\eta)}^{(\beta)}f(x)>\lambda\})\leq
(\frac{C}{\lambda} ||f||_{L^{r}(\mu)})^{nr/(n-\beta r)}.
\end{equation}

Let $E_{\lambda}=\{x:M_{r,(\eta)}^{(\beta)}f(x)>\lambda\}$. By the
definition of $M_{r,(\eta)}^{(\beta)}$, for any $x\in E_{\lambda}$,
there exists a ball $B_{x}$ containing $x$ such that
\begin{equation}
\frac{1}{\mu(\eta B_{x})^{1-\beta
r/n}}\int_{B_{x}}|f(x)|^{r}d\mu(x)\geq \lambda^{r}.
\end{equation}

Note that for $\eta\geq 5$. By Theorem 1.2 in [7], we can pick a
disjoint collection $\{B_{x_{i}}\}$ with $x_{i}\in E_{\lambda}$ and
$E_{\lambda}\subset\bigcup_{x\in
E_{\lambda}}B_{x}\subset\bigcup_{i}5B_{x_{i}}\subset\bigcup_{i}\eta
B_{x_{i}}$. Let $q=\frac{nr}{n-\beta r}$ and $r/q<1$, then
\begin{equation}
\mu(E_{\lambda})^{r/q}\leq \mu (\bigcup_{i}5B_{x_{i}})^{r/q}\leq \mu
(\bigcup_{i}\eta B_{x_{i}})^{r/q}\leq \sum_{i}\mu(\eta
B_{x_{i}})^{r/q}.
\end{equation}
Since $r/q=1-\beta r/n$, then
\begin{equation}
\sum_{i}\mu(\eta B_{x_{i}})^{r/q}\leq
\frac{1}{\lambda^{r}}\int_{X}|f|^{r}(\sum_{i}\chi_{B_{x_{i}}})d\mu.
\end{equation}
Therefore,
\begin{equation}\mu(E_{\lambda})\leq
\frac{C}{\lambda^{q}}||f||_{L^{r}(\mu)}^{q}.
\end{equation}

If $r <s< n/\beta$,
 using the H\"{o}lder's inequality, we deduce
\begin{equation}
 M_{r,(\eta)}^{(\beta)}f(x)\leq {
 M}_{s,(\eta)}^{(\beta)}f(x).
 \end{equation}
By the preceding arguments, then we obtain
\begin{equation}
\mu(E_{\lambda})\leq\mu(\{x:{
 M}_{s,(\eta)}^{(\beta)}f(x)>\lambda\})\leq
(\frac{C}{\lambda}||f||_{L^{s}(\mu)})^{ns/(n-\beta s)}.
\end{equation}
 By the Marcinkiewicz interpolation
theorem,  the proof of Lemma 2.2 is completed.
\end{proof}

With the similar method to proof of Lemma 3 in [2] or Lemma 2.1 in
[22], it is easy to obtain the following Lemma 2.3. Here we omit the
details.
\begin{lem}%lemma2.4
For $0\leq \beta<n$, we have the following properties:

 (1) If
$B\subset Q\subset R$ are balls in $X$, then $K^{(\beta)}_{B,Q}\leq
K^{(\beta)}_{B,R}$, $K^{(\beta)}_{Q,R}\leq K^{(\beta)}_{B,R}$ and
$K^{(\beta)}_{B,R}\leq C(K^{(\beta)}_{B,Q}+K^{(\beta)}_{Q,R})$

(2) If $B\subset Q$ have comparable sizes, then
$K^{(\beta)}_{B,Q}\leq C$.

(3) If N is a positive integer and the balls $6B, 6^{2}B, \cdots,
6^{N-1}B$ are non doubling balls, then $K^{(\beta)}_{B,6^{N}B}\leq
C$.
\end{lem}

Lemma 2.4 and Lemma 2.5 can be obtained by analogue to Lemma 3.11
and Lemma 3.12 in [3] or Lemma 5 and Lemma 6 in [2] respectively.
Here we omit the details of proof.
\begin{lem}%lemma2.5
For $0\leq \beta<n$, there exists a positive constant $P_{\beta}$
(big enough), depending on $\beta$, $n$ and $C_{\lambda}$, such that
if $B_{1}\subset B_{2} \subset \cdots \subset B_{m}$ are concentric
balls with $K^{(\beta)}_{B_{i},B_{i+1}}>P_{\beta}$ for $i=1, 2,
\cdots, m-1$, then there exists a positive constant $C$, depending
on $\beta$, $n$ and $C_{\lambda}$, such that
$\sum_{i=1}^{m-1}K^{(\beta)}_{B_{i},B_{i+1}}\leq C
K^{(\beta)}_{B_{1},B_{m}}$.
\end{lem}

\begin{lem}%lemma2.6
For $0\leq \beta<n$, there exists a positive constant $P'_{\beta}$
(big enough), depending on $\beta$, $n$ and $C_{\lambda}$, such that
if $x\in X$ is some fixed point and $\{f_{B}\}_{B\ni x}$ is a set of
numbers such that $|f_{B}-f_{Q}|\leq C_{x}$ for all doubling balls
$B\subset Q$ with $x\in B$ such that $K^{(\beta)}_{B,Q}\leq
P'_{\beta}$, then there exists a positive constant $C$, depending on
$\beta$, $n$ and $C_{\lambda}$, such that for all foubling balls
$B\subset Q$ with $x\in B$, $|f_{B}-f_{Q}|\leq CK^{(\beta)}_{B,Q}
C_{x}$.
\end{lem}

\begin{proof} (of Theorem1.1) %proof of Theorem 1.1
For all $1<p<n/\alpha$, we firstly establish the following sharp
maximal function estimate
\begin{equation}M^{\sharp,(\alpha)}(I_{\alpha}f)(x)\leq CM^{(\alpha)}_{r,(5)}f(x).
\end{equation}
Suppose (2.14) is valid for a moment. Choosing $r$ such that
$1<r<p<n/\alpha$ and $1/q=1/p-\alpha/n$. By Lemma 2.1 and Lemma 2.2,
we have

\begin{equation}
\begin{split}
&||I_{\alpha}f||_{L^{q}(\mu)}\leq
||N(I_{\alpha}f)||_{L^{q}(\mu)}\leq
C||M^{\sharp,(\alpha)}(I_{\alpha}f)||_{L^{q}(\mu)}\\
\leq& C||M^{(\alpha)}_{r,(5)}f||_{L^{q}(\mu)}\leq C
||f||_{L^{p}(\mu)}.
\end{split}
\end{equation}

As in the proof of Theorem 9.1 in [22], to obtain (2.14),  by Lemma
2.4 and Lemma 2.5, it suffices to show that
\begin{equation}\frac{1}{\mu(6B)}\int_{B}|I_{\alpha}f(y)-m_{B}(I_{\alpha}(f\chi_{X\backslash \frac{6}{5}B}))|d\mu(y)
\leq CM^{(\alpha)}_{r,(5)}f(x).
\end{equation}
 holds for
any $x$ and ball $B$ with $x\in B$, and
\begin{equation}
|m_{B}(I_{\alpha}(f\chi_{X\backslash
\frac{6}{5}B}))-m_{Q}(I_{\alpha}(f\chi_{X\backslash
\frac{6}{5}Q}))|\leq CK^{(\alpha)}_{B,Q}M^{(\alpha)}_{r,(5)}f(x)
\end{equation}
for all balls $B\subset Q$ with $x\in B$, where $B$ is an arbitrary
ball and $Q$ is a doubling ball.

For any ball $B$, it is easy to see that
\begin{equation}
\begin{split}
&\frac{1}{\mu(6B)}\int_{B}|I_{\alpha}f(y)-m_{B}(I_{\alpha}(f\chi_{X\backslash \frac{6}{5}B}))|d\mu(y)\\
\leq &
\frac{1}{\mu(6B)}\int_{B}|I_{\alpha}(f\chi_{\frac{6}{5}B})(y)|d\mu(y)\\
&\ \ \ +\frac{1}{\mu(6B)}\int_{B}|I_{\alpha}(f\chi_{X\backslash
\frac{6}{5}B})(y)-m_{B}(I_{\alpha}(f\chi_{X\backslash
\frac{6}{5}B}))|d\mu(y)\\
=:&I_{1}+I_{2}.
\end{split}
\end{equation}

To estimate $I_{1}$, by (i) in Definition 1.6, the properties of
function $\lambda$ and the H\"{o}lder's inequality, we get
\begin{equation}
\begin{split}
I_{1}&\leq \frac{C}{\mu(6B)}\int_{B}\int_{\frac{6}{5}B}\frac{|f(z)|}{\lambda(y,d(y,z))^{1-\alpha/n}}d\mu(z)d\mu(y)\\
&\leq \frac{C}{\mu(6B)}\int_{B}\int_{\frac{6}{5}B}\frac{|f(z)|}{\lambda(y,r_{B})^{1-\alpha/n}}d\mu(z)d\mu(y)\\
&\leq \frac{C}{\mu(6B)}\int_{B}\int_{\frac{6}{5}B}\frac{|f(z)|}{\lambda(y,r_{6B})^{1-\alpha/n}}d\mu(z)d\mu(y)\\
&\leq \frac{C}{\mu(6B)}\int_{B}\int_{\frac{6}{5}B}\frac{|f(z)|}{\mu(6B)^{1-\alpha/n}}d\mu(z)d\mu(y)\\
&\leq
\frac{C}{\mu(6B)}\int_{B}\biggl[\frac{1}{\mu(6B)^{1-\frac{\alpha
r}{n}}}\int_{\frac{6}{5}B}|f(z)|^{r}d\mu(z)\biggr]^{\frac{1}{r}}\biggl[\frac{\mu(\frac{6}{5}B)}{\mu(6B)}\biggr]^{1-\frac{1}{r}}d\mu(y)\\
&\leq CM^{(\alpha)}_{r,(5)}f(x).\\
\end{split}
\end{equation}

For $y,y_{0}\in B$, by (ii) in Definition 1.6 and the properties of
function $\lambda$ and the H\"{o}lder's inequality, we obtain
\begin{equation}
\begin{split}
&|I_{\alpha}(f\chi_{X\backslash
\frac{6}{5}B})(y)-I_{\alpha}(f\chi_{X\backslash
\frac{6}{5}B})(y_{0})|\\
\leq &C
\int_{X\backslash\frac{6}{5}B}\frac{d(y,y_{0})^{\epsilon}}{d(y,z)^{\epsilon}\lambda(y,d(y,z))^{1-\alpha/n}}|f(z)|d\mu(z)\\
\leq
&C\sum_{k=1}^{\infty}\int_{6^{k}\frac{6}{5}B\backslash6^{k-1}\frac{6}{5}B}6^{-k\epsilon}
\frac{|f(z)|d\mu(z)}{\lambda(y,r_{5\times6^{k}\frac{6}{5}B})^{1-\alpha/n}}\\
\leq &C\sum_{k=1}^{\infty}6^{-k\epsilon}\biggl[\frac{1}{\mu(5\times
6^{k}\frac{6}{5}B)^{1-\frac{\alpha
r}{n}}}\int_{6^{k}\frac{6}{5}B}|f(z)|^{r}d\mu(z)\biggr]^{\frac{1}{r}}
\biggl[\frac{\mu(6^{k}\frac{6}{5}B)}{\mu(5\times6^{k}\frac{6}{5}B)}\biggr]^{1-\frac{1}{r}}\\
\leq& CM^{(\alpha)}_{r,(5)}f(x).\\
\end{split}
\end{equation}

 Taking the mean over $y_{0}\in B$,
then we have
\begin{equation}
I_{2}\leq CM^{(\alpha)}_{r,(5)}f(x).
\end{equation}
So (2.16) holds from (2.18) to (2.21).

 Now we prove (2.17). Consider
two balls $B\subset Q$ with $x\in B$, where $B$ is an arbitrary ball
and $Q$ is a doubling ball. Let $N=N_{B,Q}+1$, then we have
\begin{equation}
\begin{split}
&\biggl|m_{B}\biggr[I_{\alpha}(f\chi_{X\backslash\frac{6}{5}B})\biggr]
-m_{Q}\biggl[I_{\alpha}(f\chi_{X\backslash\frac{6}{5}Q})\biggr]\biggr|\\
\leq &
\biggl|m_{B}\biggl[I_{\alpha}(f\chi_{X\backslash6^{N}B})\biggr]
-m_{Q}\biggl[I_{\alpha}(f\chi_{X\backslash6^{N}B})\biggr]\biggr|\\
&\ \ +\biggl|m_{B}\biggl[I_{\alpha}
(f\chi_{6^{N}B\backslash\frac{6}{5}B})\biggr]\biggr|+\biggl|m_{Q}\biggl[I_{\alpha}
(f\chi_{6^{N}B\backslash\frac{6}{5}Q})\biggr]\biggr|\\=&:J_{1}+J_{2}+J_{3}.
\end{split}
\end{equation}

With the same method to estimate $I_{2}$, we immediately get
\begin{equation}J_{1}\leq C
M^{(\alpha)}_{r,(5)}f(x). \end{equation}

To estimate $J_{2}$, for $z\in B$, it is easy to see that
\begin{equation}|I_{\alpha}(f\chi_{6^{N}B\backslash\frac{6}{5}B})(z)|\leq |I_{\alpha}(f\chi_{6^{N}B\backslash6B})(z)|
+|I_{\alpha}(f\chi_{6B\backslash\frac{6}{5}B})(z)|.\end{equation}

By (i) in Definition 1.6 and the H\"{o}lder's inequality, we deduce
\begin{equation}
\begin{split}
&|I_{\alpha}(f\chi_{6^{N}B\backslash6B})(z)|\\
\leq&
C\int_{6^{N}B\backslash6B}\frac{|f(y)|}{\lambda(z,d(z,y))^{1-\alpha/n}}d\mu(y)\\
\leq&C \sum_{k=1}^{N-1}\int_{6^{k+1}B\backslash6^{k}B}\frac{|f(y)|}{\lambda(z,r_{6^{k+1}B})^{1-\frac{\alpha}{n}}}d\mu(y)\\
\leq&C \sum_{k=1}^{N-1}[\frac{\mu(5\times
6^{k+1}B)}{\lambda(z,r_{5\times6^{k+1}B})}]^{1-\frac{\alpha}{n}}\frac{1}{\mu(5\times6^{k+1}B)^{1-\frac{\alpha}{n}}}\int_{6^{k+1}B}|f(y)|d\mu(y)\\
\leq&C
\sum_{k=1}^{N-1}[\frac{\mu(5\times6^{k+1}B)}{\lambda(z,r_{5\times6^{k+1}B})}]^{1-\frac{\alpha}{n}}\biggl[\frac{1}{\mu(5\times
6^{k+1}\frac{6}{5}B)^{1-\frac{\alpha
r}{n}}}\int_{6^{k+1}\frac{6}{5}B}|f(y)|^{r}d\mu(y)\biggr]^{\frac{1}{r}}\\
&\ \times\biggl[\frac{\mu(6^{k+1}\frac{6}{5}B)}{\mu(5\times6^{k+1}\frac{6}{5}B)}\biggr]^{1-\frac{1}{r}}\\
\leq& CK^{(\alpha)}_{B,Q}M^{(\alpha)}_{r,(5)}f(x).\\
\end{split}
\end{equation}
Also, we obtain that
\begin{equation}
\begin{split}
&|I_{\alpha}(f\chi_{6B\backslash\frac{6}{5}B})(z)|\\
\leq &C\int_{6B\backslash\frac{6}{5}B}\frac{|f(y)|}{\lambda(z,r_{B})^{1-\alpha/n}}d\mu(y)\\
\leq &\frac{C}{\mu(5\times 6B)^{1-\alpha/n}}\int_{6B}|f(y)|d\mu(y)\\
\leq &C\biggl[\frac{1}{\mu(5\times 6B)^{1-\frac{\alpha
r}{n}}}\int_{6B}|f(y)|^{r}d\mu(y)\biggr]^{\frac{1}{r}}\biggl[\frac{\mu(6B)}{\mu(5\times6B)}\biggr]^{1-\frac{1}{r}}\\
\leq& CM^{(\alpha)}_{r,(5)}f(x).\\
\end{split}
\end{equation}
Then
\begin{equation}|I_{\alpha}(f\chi_{6^{N}B\backslash\frac{6}{5}B})(z)|\leq CK^{(\alpha)}_{B,Q}M^{(\alpha)}_{r,(5)}f(x). \end{equation}
Taking mean for $z$ over $B$, we have
\begin{equation}J_{2}\leq CK^{(\alpha)}_{B,Q}M^{(\alpha)}_{r,(5)}f(x). \end{equation}

Similarly, we also obtain that
\begin{equation}J_{3}\leq CK^{(\alpha)}_{B,Q}M^{(\alpha)}_{r,(5)}f(x). \end{equation}
From (2.22) to (2.29), we yield (2.17) holds. Hence the proof of
Theorem 1.1 is completed.
\end{proof}

\section{Proof of Theorem 1.2}
Let us firstly give the equivalent definition of $RBMO(\mu)$, which
is useful in proving Theorem 1.2.

\noindent{\bf Definition 3.1.}$^{[8]}$\quad Let $\rho>1$ be some
fixed constant. A function $b\in L_{loc}^{1}(\mu)$ is said to belong
to $RBMO(\mu)$ if there exists a constant $C
>0$ such that for any ball $B$, a number $b_{B}$ such that
\begin{equation}
\frac{1}{\mu(\rho B)}\int_{B}|b(x)-b_{B}|d\mu(x)\leq C,
\end{equation}
and for any two balls $B\subset Q$,
 \begin{equation}
  |b_{B}-b_{Q}|\leq CK_{B,Q}.
\end{equation}
The minimal constant $C$ appearing in (3.1) and (3.2) is defined to
be the $RBMO(\mu)$ norm of $f$ and denoted by $||b||_{\ast}$. The
norm $||b||_{\ast}$ is independent of $\rho>1$.

\begin{lem}$^{[22]}$%lemma3.1
For any ball $B$, we have
\begin{equation}
|b_{B}-b_{6^{k}\frac{6}{5}B}|\leq C k||b||_{\ast}.\end{equation}
\end{lem}

\begin{proof} (of Theorem1.2) %proof of Theorem 1.2
  For all $1<p<n/\alpha$, we firstly establish the following sharp
maximal function estimate
\begin{equation}M^{\sharp,(\alpha)}([b,I_{\alpha}]f)(x)\leq
C||b||_{\ast}[
M^{(\alpha)}_{r,(5)}f(x)+M_{r,(6)}(I_{\alpha}f)(x)+I_{\alpha}(|f|)(x)].
\end{equation}
Suppose (3.4) holds for a moment. By Lemma 3.3 in [22], we can
assume that $b\in L^{\infty}(\mu)$. Choosing $r$ such that
$1<r<p<n/\alpha$ and $1/q=1/p-\alpha/n$. By Lemma 2.1, Lemma 2.2 and
Theorem 1.1 in this paper, we obtain

\begin{equation}
\begin{split} &||[b,I_{\alpha}]f||_{L^{q}(\mu)}\leq
||N([b,I_{\alpha}]f)||_{L^{q}(\mu)}\leq
C||M^{\sharp,(\alpha)}([b,I_{\alpha}]f)||_{L^{q}(\mu)}\\
\leq
C&\{||M^{(\alpha)}_{r,(5)}f||_{L^{q}(\mu)}+||M_{r,(6)}(I_{\alpha}f)||_{L^{q}(\mu)}+||I_{\alpha}(|f|)||_{L^{q}(\mu)}\}\\
\leq C& \{||f||_{L^{p}(\mu)}+||I_{\alpha}f||_{L^{q}(\mu)}\}\leq C
||f||_{L^{p}(\mu)}.
\end{split}
\end{equation}

Let $\{b_{B}\}$ be a collection of numbers satisfying for ball $B$,
\begin{equation}
\int_{B}|b-b_{B}|d\mu\leq 2\mu(6B)||b||_{\ast},
\end{equation}
and for two balls $B\subset Q$,
\begin{equation}
|b_{B}-b_{Q}|\leq 2K_{B,Q}||b||_{\ast}.
\end{equation}

As in the proof of Theorem 1 in [2], to obtain (3.4), by Lemma 2.4
and Lemma 2.5 in Section 2 of this paper,  it suffices to deduce
that
\begin{equation}\frac{1}{\mu(6B)}\int_{B}|[b,I_{\alpha}](f)(y)-h_{B}|d\mu(y)
\leq C||b||_{\ast}(
M^{(\alpha)}_{r,(5)}f(x)+M_{r,(6)}(I_{\alpha}f)(x)).
\end{equation}
 holds for
any $x$ and ball $B$ with $x\in B$, and
\begin{equation}
|h_{B}-h_{Q}|\leq
C||b||_{\ast}K_{B,Q}K^{(\alpha)}_{B,Q}(M^{(\alpha)}_{r,(5)}f(x)+I_{\alpha}(|f|)(x))
\end{equation}
for all balls $B\subset Q$ with $x\in B$, where $B$ is an arbitrary
ball, $Q$ is a doubling ball and for any ball $B$, we denote
$$h_{B}:=
m_{B}(I_{\alpha}[(b-b_{B})f\chi_{X\backslash\frac{6}{5}B}]).$$

To obtain (3.8), we write $[b,I_{\alpha}]f$ as follows.
\begin{equation}
\begin{split}
 &[b,I_{\alpha}]f(y)=(b(y)-b_{B})I_{\alpha}f(y)-I_{\alpha}((b-b_{B})f)(y)\\
=&(b(y)-b_{B})I_{\alpha}f(y)-I_{\alpha}((b-b_{B})f_{1})(y)-I_{\alpha}((b-b_{B})f_{2})(y),
\end{split}
\end{equation}
where $f_{1}=f\chi_{\frac{6}{5}B}$ and $f_{2}=f-f_{1}$. Now, by the
H\"{o}lder's inequality, we have
\begin{equation}
\begin{split}&\frac{1}{\mu(6B)}\int_{B}|(b(y)-b_{B})I_{\alpha}f(y)|d\mu(y)\\
\leq &(\frac{1}{\mu(6B)}\int_{B}|b(y)-b_{B}|^{r'}d\mu(y))^{1/r'}
(\frac{1}{\mu(6B)}\int_{B}|I_{\alpha}f(y)|^{r}d\mu(y))^{1/r}\\
\leq &C||b||_{\ast}M_{r,(6)}(I_{\alpha}f)(x).
\end{split}
\end{equation}

We take $s=\sqrt{r}$ and let $1/t=1/s-\alpha/n$. Using the
H\"{o}lder's inequality, the result of Theorem 1.1 and Definition
3.1, we obtain

\begin{equation}
\begin{split}
&\frac{1}{\mu(6B)}\int_{B}|I_{\alpha}((b-b_{B})f_{1})(y)|d\mu(y)\\
\leq&
\frac{\mu(B)^{1-1/t}}{\mu(6B)}||I_{\alpha}((b-b_{B})f_{1})||_{L^{t}(\mu)}\\
\leq&
\frac{\mu(B)^{1-1/t}}{\mu(6B)}||((b-b_{B})f_{1})||_{L^{s}(\mu)}\\
\leq&
\biggl(\frac{1}{\mu(6B)}\int_{\frac{6}{5}B}|b-b_{B}|^{ss'}d\mu(y)\biggr)^{\frac{1}{ss'}}\\
&\ \times\biggl(\frac{1}{\mu(6B)^{1-\alpha
r/n}}\int_{\frac{6}{5}B}|f(y)|^{r}d\mu(y)\biggr)^{\frac{1}{r}}\\
\leq& C||b||_{\ast}M^{(\alpha)}_{r,(5)}f(x).
\end{split}
\end{equation}

Next, to prove (3.8), we only need to compute
$|I_{\alpha}((b-b_{B})f_{2})(y)-h_{B}|$. For $y,\ y_{0}\in B$, by
Lemma 3.1, we have
\begin{equation}
\begin{split}
&|I_{\alpha}((b-b_{B})f_{2})(y)-I_{\alpha}((b-b_{B})f_{2})(y_{0})|\\
\leq &C\int_{X\backslash\frac{6}{5}B}
\frac{d(y,y_{0})^{\epsilon}}{d(y,z)^{\epsilon}\lambda(y,d(y,z))^{1-\frac{\alpha}{n}}}|b(z)-b_{B}||f(z)|d\mu(z)\\
\leq
&C\sum_{k=1}^{\infty}\int_{6^{k}\frac{6}{5}B\backslash6^{k-1}\frac{6}{5}B}
\frac{6^{-k\epsilon}}{\lambda(y,r_{6^{k}\frac{6}{5}B})^{1-\frac{\alpha}{n}}}\\
&\ \ \times(|b(z)-b_{6^{k}\frac{6}{5}B}|+|b_{B}-b_{6^{k}\frac{6}{5}B}|)|f(z)|d\mu(z)\\
\leq&
C\sum_{k=1}^{\infty}6^{-k\epsilon}\frac{1}{\mu(5\times6^{k}\frac{6}{5}B
)^{1-\frac{\alpha}{n}}}\int_{6^{k}\frac{6}{5}B}|b(z)-b_{6^{k}\frac{6}{5}B}|
|f(z)|d\mu(z)\\
&+C\sum_{k=1}^{\infty}k6^{-k\epsilon}||b||_{\ast}\frac{1}{\mu(5\times6^{k}\frac{6}{5}B
)^{1-\frac{\alpha}{n}}}\int_{6^{k}\frac{6}{5}B}|f(z)|d\mu(z)\\
\leq
&C\sum_{k=1}^{\infty}6^{-k\epsilon}\biggl[\frac{1}{\mu(5\times6^{k}\frac{6}{5}B
)^{1-\frac{\alpha r}{n}}}\int_{6^{k}\frac{6}{5}B}|f(z)|^{r}d\mu(z)\biggr]^{1/r}\\
&\times\biggl[\frac{1}{\mu(5\times6^{k}\frac{6}{5}B
)}\int_{6^{k}\frac{6}{5}B}|b(z)-b_{6^{k}\frac{6}{5}B}|^{r'}d\mu(z)\biggr]^{1/r'}\\
&+C\sum_{k=1}^{\infty}k6^{-k\epsilon}||b||_{\ast}\biggl[\frac{1}{\mu(5\times6^{k}\frac{6}{5}B
)^{1-\frac{\alpha
r}{n}}}\int_{6^{k}\frac{6}{5}B}|f(z)|^{r}d\mu(z)\biggr]^{\frac{1}{r}}\\
&\ \times\biggl[\frac{\mu(6^{k}\frac{6}{5}B)}{\mu(5\times6^{k}\frac{6}{5}B)}\biggr]^{\frac{1}{r'}}\\
\leq &C||b||_{\ast}M^{(\alpha)}_{r,(5)}f(x).
\end{split}
\end{equation}
Taking the mean over $y_{0}\in B$, we obtain
\begin{equation}
|I_{\alpha}((b-b_{B})f_{2})(y)-h_{B}|\leq
C||b||_{\ast}M^{(\alpha)}_{r,(5)}f(x).
\end{equation}
So (3.8) is proved.

To prove (3.9), we consider two balls $B\subset Q$ with $x\in B$,
where $B$ is an arbitrary ball, $Q$ is a doubling ball. Denote
$N=N_{B,Q}+1$, we write
\begin{equation}
\begin{split}
&|m_{B}(I_{\alpha}((b-b_{B})f\chi_{X\backslash\frac{6}{5}B}))-m_{Q}(I_{\alpha}((b-b_{Q})f\chi_{X\backslash\frac{6}{5}Q}))|\\
\leq & |m_{B}(I_{\alpha}((b-b_{B})f\chi_{6B\backslash\frac{6}{5}B}))|\\
&\ +|m_{B}(I_{\alpha}((b_{B}-b_{Q})f\chi_{X\backslash6B}))|\\
&\ +|m_{B}(I_{\alpha}((b-b_{Q})f\chi_{6^{N}B\backslash 6B}))|\\
&\
+|m_{B}(I_{\alpha}((b-b_{Q})f\chi_{X\backslash6^{N}B}))-m_{Q}(I_{\alpha}((b-b_{Q})f\chi_{X\backslash6^{N}B}))|\\
&\ +|m_{Q}(I_{\alpha}((b-b_{Q})f\chi_{6^{N}B\backslash\frac{6}{5}Q}))|\\
=:&L_{1}+L_{2}+L_{3}+L_{4}+L_{5}.
\end{split}
\end{equation}
For $y\in B$, by the H\"{o}lder's inequality, we get
\begin{equation}
\begin{split} &|I_{\alpha}((b-b_{B})f\chi_{6B\backslash\frac{6}{5}B})(y)|\\
\leq&
C\int_{6B}\frac{|b(z)-b_{B}||f(z)|}{\lambda(y,d(y,z))^{1-\alpha/n}}d\mu(z)\\
\leq&C
 \biggl[\frac{1}{\mu(5\times6B
)^{1-\frac{\alpha
r}{n}}}\int_{6B}|f(z)|^{r}d\mu(z)\biggr]^{\frac{1}{r}}\\
&\ \times\biggl[\frac{1}{\mu(5\times6B
)}\int_{6B}|b(z)-b_{B}|^{r'}d\mu(z)\biggr]^{\frac{1}{r'}}\\
\leq & C||b||_{\ast}M^{(\alpha)}_{r,(5)}f(x).
\end{split}
\end{equation}
Then we get $L_{1} \leq C||b||_{\ast}M^{(\alpha)}_{r,(5)}f(x).$

For $x,\ y\in B$, it is easy to see that
\begin{equation}
|I_{\alpha}(f\chi_{X\backslash6B})(y)|\leq
I_{\alpha}(|f|)(x)+CM^{(\alpha)}_{r,(5)}f(x).
\end{equation}
Then, by Definition 3.1, we get
\begin{equation}
L_{2} \leq CK_{B,Q}||b||_{\ast}(I_{\alpha}(|f|)(x)+
M^{(\alpha)}_{r,(5)}f(x)).
\end{equation}

Let us estimate $L_{3}$. For $y\in B$, we have
\begin{equation}
\begin{split}&|I_{\alpha}((b-b_{Q})f\chi_{6^{N}B\backslash6B})(y)| \\
\leq &C\sum_{k=1}^{N-1}\int_{6^{k+1}B\backslash
6^{k}B}\frac{|b(z)-b_{Q}||f(z)|}{\lambda(y,d(y,z))^{1-\alpha/n}}d\mu(z)\\
\leq & C\sum_{k=1}^{N-1}[\frac{\mu(5\times 6^{k+1}B)}{\lambda
(y,r_{5\times6^{k+1}B})}]^{1-\alpha/n}\\
&\times\frac{1}{\mu(5\times
6^{k+1}B)^{1-\alpha/n}}\int_{6^{k+1}B}|b(z)-b_{Q}||f(z)|d\mu(z)\\
\leq & CK^{(\alpha)}_{B,Q}\biggl[\frac{1}{\mu(5\times6^{k+1}B
)^{1-\frac{\alpha
r}{n}}}\int_{6^{k+1}B}|f(z)|^{r}d\mu(z)\biggr]^{\frac{1}{r}}\\
&\times\biggl[\frac{1}{\mu(5\times6^{k+1}B
)}\int_{6^{k+1}B}|b(z)-b_{Q}|^{r'}d\mu(z)\biggr]^{\frac{1}{r'}}\\
\leq &
CK^{(\alpha)}_{B,Q}M^{\alpha}_{r,(5)}f(x)\biggl[\frac{1}{\mu(5\times6^{k+1}B
)}\int_{6^{k+1}B}|b(z)-b_{6^{k+1}B}\\
&+b_{6^{k+1}B}-b_{Q}|^{r'}d\mu(z)\biggr]^{\frac{1}{r'}}\\
\leq &
CK_{B,Q}K^{(\alpha)}_{B,Q}||b||_{\ast}M^{(\alpha)}_{r,(5)}f(x).
\end{split}
\end{equation}
Here we have used the fact that $|b_{6^{k+1}B}-b_{Q}|\leq C
K_{B,Q}||b||_{\ast}$.

Taking the mean over $B$, we have
$$L_{3} \leq CK_{B,Q}K^{(\alpha)}_{B,Q}||b||_{\ast}M^{(\alpha)}_{r,(5)}f(x).$$

For $L_{4}$. Operating as in (3.13), for any $y\in B$ and $z\in Q$,
we obtain
$$|T((b-b_{Q})f\chi_{X\backslash 6^{N}B})(y)-T((b-b_{Q})f\chi_{X\backslash 6^{N}B})(z)|
\leq C||b||_{\ast}M^{(\alpha)}_{r,(5)}f(x).$$ Taking the mean over
$B$ for $y$ and over $Q$ for $z$, we have
$$L_{4}\leq C||b||_{\ast}M^{(\alpha)}_{r,(5)}f(x).$$

For $L_{5}$, similar to $L_{1}$, we can deduce $L_{5}\leq
C||b||_{\ast}M^{(\alpha)}_{r,(5)}f(x).$ Thus (3.9) is valid and the
proof of Theorem 1.2 is finished .
\end{proof}

\section{Proof of Theorem 1.3}

To prove Theorem 1.3, we need the following some lemmas.
\begin{lem}$^{[8,22]}$%lemma4.1
Let $1\leq p<\infty$ and $1<\rho <\infty$. If $b\in RBMO(\mu)$, then
for any ball $B\in X$,
\begin{equation}
\biggl\{\frac{1}{\mu(\rho
B)}\int_{B}|b(x)-m_{\widetilde{B}}(b)|^{p}d\mu(x)\biggr\}^{1/p}\leq
C||b||_{\ast}.\end{equation}
\end{lem}

\begin{lem}$^{[9]}$%lemma4.2
For any ball $B$, we have
\begin{equation}
|m_{\widetilde{6^{j}\frac{6}{5}B}}(b)-m_{\widetilde{B}}(b)|\leq
Cj||b||_{\ast}.
\end{equation}
\end{lem}

\begin{proof} (of Theorem 1.3) %proof of Theorem 1.3
 We prove the theorem by induction on $k$.
If $k=1$, the result of Theorem 1.2 asserts that $[b,I_{\alpha}]$ is
bounded from $L^{p}(\mu)$ to $L^{q}(\mu)$ for any $1<p<n/\alpha$,
$0<\alpha <n$ and $1/q=1/p-\alpha/n$. Now we assume that $k\geq 2$
is an integer and that for any $1\leq i\leq k-1$ and any subset
$\sigma=\{\sigma(1),\cdot\cdot\cdot,\sigma(i)\}$ of
$\{1,\cdot\cdot\cdot,k\}$, $I_{\alpha,\vec{b}_{\sigma}}$ is bounded
from $L^{p}(\mu)$ to $L^{q}(\mu)$ for any for any $1<p<n/\alpha$,
$0<\alpha <n$ and $1/q=1/p-\alpha/n$. We next claim that for any
$1<r<\infty$, $I_{\alpha,\vec{b}}$ satisfies the following sharp
maximal function estimate
\begin{equation}
\begin{split}M^{\sharp,(\alpha)}(I_{\alpha,\vec{b}}f)(x) &\leq
C||\vec{b}||_{\ast}\{M_{r,(6)}(I_{\alpha}f)(x)+M^{(\alpha)}_{r,(5)}f(x)\}\\
&\ +C\sum_{i=1}^{k-1}\sum_{\sigma\in
C_{i}^{k}}||\vec{b}_{\sigma}||_{\ast}M_{r,(6)}(I_{\alpha,\vec{b}_{\sigma'}}f)(x).
\end{split}
\end{equation}

Suppose (4.3) holds for a moment.  Now we prove $T_{\vec{b}}$
satisfies (1.14). By Lemma 3.3 in [22], we can assume that $b_{i}\in
L^{\infty}(\mu)$ for $1\leq i\leq k$. Choosing $r$ such that
$1<r<p<n/\alpha$ and $1/q=1/p-\alpha/n$. By Lemma 2.1, Lemma2.2,
Theorem 1.1 and  Theorem 1.2 in this paper, we deduce
\begin{equation}
\begin{split} &||I_{\alpha,\vec{b}}f||_{L^{q}(\mu)}\leq
C||N(I_{\alpha,\vec{b}}f)||_{L^{q}(\mu)}
\leq C||M^{\sharp,(\alpha)}(I_{\alpha,\vec{b}}f)||_{L^{q}(\mu)}\\
\leq& C||\vec{b}||_{\ast}\{||M_{r,(6)}(I_{\alpha}f)||_{L^{q}(\mu)}+||M^{(\alpha)}_{r,(5)}f||_{L^{q}(\mu)}\}\\
&\ +C\sum_{i=1}^{k-1}\sum_{\sigma\in
C_{i}^{k}}||\vec{b}_{\sigma}||_{\ast}
||M_{r,(6)}(I_{\alpha,\vec{b}_{\sigma'}}f)||_{L^{q}(\mu)}\\
\leq&
C||\vec{b}||_{\ast}(||I_{\alpha}f||_{L^{q}(\mu)}+||f||_{L^{p}(\mu)})+C\sum_{i=1}^{k-1}\sum_{\sigma\in
C_{i}^{k}}||\vec{b}_{\sigma}||_{\ast}||I_{\alpha,\vec{b}_{\sigma'}}f||_{L^{q}(\mu)}\\
\leq& C||\vec{b}||_{\ast}||f||_{L^{p}(\mu)}
\end{split}
\end{equation}

As in the proof of Theorem 2 in [9], to obtain (4.3), by Lemma 2.4
and Lemma 2.5 in this paper,  it only need to show that
\begin{equation}
\begin{split}
&\frac{1}{\mu(6B)}\int_{B}|I_{\alpha,\vec{b}}f(y)-h_{B}|d\mu(y)\\
\leq&
C||\vec{b}||_{\ast}\{M_{r,(6)}(I_{\alpha}f)(x)+M^{(\alpha)}_{r,(5)}f(x)\}
 +C\sum_{i=1}^{k-1}\sum_{\sigma\in
C_{i}^{k}}||\vec{b}_{\sigma}||_{\ast}M_{r,(6)}(I_{\alpha,\vec{b}_{\sigma'}}f)(x)
\end{split}
\end{equation}
hold for all $x$ and $B$ with $x\in B$, and
\begin{equation}
\begin{split} |h_{B}-h_{Q}|&\leq
C(K_{B,Q})^{k}K^{(\alpha)}_{B,Q}||\vec{b}||_{\ast}\{M_{r,(6)}(I_{\alpha}f)(x)+M^{(\alpha)}_{r,(5)}f(x)\}\\
&\ +C(K_{B,Q})^{k}\sum_{i=1}^{k-1}\sum_{\sigma\in
C_{i}^{k}}||\vec{b}_{\sigma}||_{\ast}M_{r,(6)}(I_{\alpha,\vec{b}_{\sigma'}}f)(x)
\end{split}
\end{equation}
holds for any cube $B\subset Q$ with $x\in B$, where $B$ is an
arbitrary cube and $Q$ is a doubling cube. We denote
$$h_{B}=m_{B}\left(I_{\alpha}\left[\left(b_{1}-m_{\widetilde{B}}(b_{1})\right)\cdot\cdot\cdot \left(b_{k}-m_{\widetilde{B}}(b_{k})\right)
f\chi_{X\backslash\frac{6}{5}B}\right]\right),$$ and

$$h_{R}=m_{R}\left(I_{\alpha}\left[\left(b_{1}-m_{R}(b_{1})\right)\cdot\cdot\cdot \left(b_{k}-m_{R}(b_{k})\right)
f\chi_{X\backslash\frac{6}{5}R}\right]\right).$$

Let us firstly estimate (4.5). It is easy to see that
\begin{equation}\prod_{i=1}^{k}\left[b_{i}(z)-m_{\widetilde{B}}(b_{i})\right]=\sum_{i=0}^{k}\sum_{\sigma\in
C_{i}^{k}}[b(z)-b(y)]_{\sigma'}[b(y)-m_{\widetilde{B}}(b)]_{\sigma}
\end{equation} for $y,z\in X$, where if $i=0$, then
$\sigma'=\{1,2\cdot\cdot\cdot k\}$, $\sigma=\emptyset$ and
$[b(y)-m_{\widetilde{B}}(b)]_{\emptyset}=1$. Hence
$$I_{\alpha,\vec{b}}f(y)=I_{\alpha}\left(\prod_{i=1}^{k}\left[b_{i}-m_{\widetilde{B}}(b_{i})\right]f\right)(y)-\sum_{i=1}^{k}\sum_{\sigma\in
C_{i}^{k}}\left[b(y)-m_{\widetilde{B}}(b)\right]_{\sigma}I_{\alpha,\vec{b}_{\sigma'}}f(y),$$
and if $i=k$, we denote $I_{\alpha}f(y)$ by
$I_{\alpha,\vec{b}_{\sigma'}}f(y)$. Thus,
\begin{equation}
\begin{split}
&\frac{1}{\mu(6B)}\int_{B}|I_{\alpha,\vec{b}}f(y)-h_{B}|d\mu(y)\\
\leq
&\frac{1}{\mu(6B)}\int_{B}\left|I_{\alpha}\left(\prod_{i=1}^{k}\left[b_{i}-m_{\widetilde{B}}(b_{i})\right]
f\chi_{\frac{6}{5}B}\right)(y)\right|d\mu(y)\\
\ &+\sum_{i=1}^{k}\sum_{\sigma\in
C_{i}^{k}}\frac{1}{\mu(6B)}\int_{B}\left|\left[b(y)-m_{\widetilde{B}}(b)\right]_{\sigma}\right
|\left|I_{\alpha,\vec{b}_{\sigma'}}f(y)\right|d\mu(y)\\
\ &+\frac{1}{\mu(6B)}\int_{B}\left|I_{\alpha}
\left(\prod_{i=1}^{k}\left[b_{i}-m_{\widetilde{B}}(b_{i})\right]f\chi_{X\backslash\frac{6}{5}B}\right)(y)-h_{B}\right|d\mu(y)\\
=:&II_{1}+II_{2}+II_{3}.
\end{split}
\end{equation}

Write
$$b_{i}(y)-m_{\widetilde{Q}}(b_{i})=b_{i}(y)-m_{\widetilde{\frac{4}{3}Q}}(b_{i})
+m_{\widetilde{\frac{4}{3}Q}}(b_{i})-m_{\widetilde{Q}}(b_{i})$$ for
$i=1,\cdot\cdot\cdot,k$. By Lemma 4.1, we obtain
\begin{equation}\int_{B}\prod_{i=1}^{k}|b_{i}(y)-m_{\widetilde{B}}(b_{i})|^{ss'}d\mu(y)\leq C||\vec{b}||_{\ast}^{ss'}
\mu(6B).\end{equation}
 Take $s=\sqrt{r}$ and let $1/t=1/s-\alpha/n$.
By Theorem 1.1, the H\"{o}lder's inequality and (4.9), we deduce
\begin{equation}
\begin{split} II_{1}&\leq
\frac{\mu(B)^{1-1/t}}{\mu(6B)}\left|\left|I_{\alpha}\left(\prod_{i=1}^{k}\left[
b_{i}-m_{\tilde{B}}(b_{i})\right]f\chi_{\frac{6}{5}B}\right)\right|\right|_{L^{t}(\mu)}\\
&\leq
C\frac{\mu(B)^{1-1/t}}{\mu(6B)}\left|\left|\prod_{i=1}^{k}\left[
b_{i}-m_{\tilde{B}}(b_{i})\right]f\chi_{\frac{6}{5}B}\right|\right|_{L^{s}(\mu)}\\
&\leq C\left(\frac{1}{\mu(6B)}\int_{\frac{6}{5}B}\prod_{i=1}^{k}
\left|b_{i}(y)-m_{\tilde{B}}(b_{i})\right|^{ss'}d\mu(y)\right)^{\frac{1}{ss'}}\\
&\ \ \times\left(\frac{1}{\mu(6B)^{1-\frac{\alpha r}{n}}}\int_{\frac{6}{5}B}|f(y)|^{r}d\mu(y)\right)^{\frac{1}{r}}\\
&\leq C||\vec{b}||_{\ast}M^{(\alpha)}_{r,(5)}f(x).
\end{split}
\end{equation}

From the H\"{o}lder's inequality and Lemma 4.1, it follows that
\begin{equation}
\begin{split} II_{2}& \leq \sum_{i=1}^{k}\sum_{\sigma\in
C_{i}^{k}}\left(\frac{1}{\mu(6B)}\int_{B}\left|\left[b(y)-m_{\tilde{B}}(b)\right]_{\sigma}\right|^{r'}d\mu(y)
\right)^{\frac{1}{r'}}\\
&\ \ \times\left(\frac{1}{\mu(6B)}\int_{B}|I_{\alpha,\vec{b}_{\sigma'}}f(y)|^{r}d\mu(y)\right)^{\frac{1}{r}}\\
&\leq C \sum_{i=1}^{k}\sum_{\sigma\in
C_{i}^{k}}||\vec{b}_{\sigma}||_{\ast}M_{r,(6)}(I_{\alpha,\vec{b}_{\sigma'}}f)(x)\\
&=C \sum_{i=1}^{k-1}\sum_{\sigma\in
C_{i}^{k}}||\vec{b}_{\sigma}||_{\ast}M_{r,(6)}(I_{\alpha,\vec{b}_{\sigma'}}f)(x)+C||b||_{\ast}M_{r,(6)}(I_{\alpha}f)(x).
\end{split}
\end{equation}

Let us estimate $II_{3}$. For $y,\ y_{0}\in B$, by the condition
(ii) in Definition 1.6 and the H\"{o}lder's inequality, we have
\begin{equation}
\begin{split} &\biggl|I_{\alpha}
\left(\prod_{i=1}^{k}\left[b_{i}-m_{\tilde{B}}(b_{i})\right]f\chi_{X\backslash\frac{6}{5}B}\right)(y)\\
&-I_{\alpha}
\left(\prod_{i=1}^{k}\left[b_{i}-m_{\tilde{B}}(b_{i})\right]f\chi_{X\backslash\frac{6}{5}B}\right)(y_{0})\biggr|\\
\leq
&C\int_{X\backslash\frac{6}{5}B}\frac{d(y,y_{0})^{\epsilon}}{d(y,z)^{\epsilon}\lambda(y,d(y,z))^{1-\frac{\alpha}{n}}}\prod_{i=1}^{k}
\left|b_{i}(z)-m_{\widetilde{B}}(b_{i})\right||f(z)|d\mu(z)\\
\leq &C\sum_{j=1}^{\infty}\int_{6^{k}\frac{6}{5}B}
6^{-j\epsilon}\frac{1}{\lambda(y,r_{6^{j}\frac{6}{5}B})^{1-\frac{\alpha}{n}}}\prod_{i=1}^{k}
\biggl(\left|b_{i}(z)-m_{\widetilde{6^{j}\frac{6}{5}B}}(b_{i})\right|\\
&\ \ +\biggl|m_{\widetilde{6^{j}\frac{6}{5}B}}(b_{i})-m_{\tilde{B}}(b_{i})\biggr|\biggr)|f(z)|d\mu(z)\\
\leq &C\sum_{j=1}^{\infty}\sum_{i=0}^{k}\sum_{\sigma\in
C_{i}^{k}}6^{-j\epsilon}j^{k-i}||\vec{b}_{\sigma'}||_{\ast}\\&\ \
\times\frac{1}{\mu(5\times6^{j}\frac{6}{5}B
)^{1-\alpha/n}}\int_{6^{j}\frac{6}{5}B}
\left|\left[b(z)-m_{\widetilde{6^{j}\frac{6}{5}B}}(b_{i})\right]_{\sigma}\right||f(z)|d\mu(z)\\
\leq &C\sum_{i=0}^{k}\sum_{\sigma\in
C_{i}^{k}}\sum_{j=1}^{\infty}6^{-j\epsilon}j^{k-i}||\vec{b}_{\sigma}||_{\ast}||\vec{b}_{\sigma'}||_{\ast}M^{(\alpha)}_{r,(5)}f(x)\\
\leq &C||\vec{b}||_{\ast}M^{(\alpha)}_{r,(5)}f(x).\\
\end{split}
\end{equation}

From the above estimate and the definition of $h_{B}$, we have
\begin{equation}
\begin{split}&\left|I_{\alpha}
\left(\prod_{i=1}^{k}\left[b_{i}-m_{\tilde{B}}(b_{i})\right]f\chi_{X\backslash\frac{6}{5}B}\right)(y)-h_{B}\right|\\
=&\biggl|I_{\alpha}
\left(\prod_{i=1}^{k}\left[b_{i}-m_{\tilde{Q}}(b_{i})\right]f\chi_{X\backslash\frac{6}{5}B}\right)(y)\\
&-m_{B}\left[I_{\alpha}
\left(\prod_{i=1}^{k}\left[b_{i}-m_{\tilde{Q}}(b_{i})\right]f\chi_{X\backslash\frac{6}{5}B}\right)\right]\biggr|\\
\leq& C||\vec{b}||_{\ast}M^{(\alpha)}_{r,(5)}f(x).\\
\end{split}
\end{equation}
Then
$$II_{3}\leq C||\vec{b}||_{\ast}M^{(\alpha)}_{r,(5)}f(x).$$
The estimate for $II_{1}$, $II_{2}$ and $II_{3}$ yields (4.5).

  Now we turn to the estimate for (4.6). For any balls
$B\subset Q$ with $x\in B$ and $Q$ is a doubling ball, we denote
$N_{B,Q}+1$ simply by $N$.

\begin{equation}
\begin{split} &\biggl| m_{B}\biggl[I_{\alpha}
\left(\prod_{i=1}^{k}\left[b_{i}-m_{\tilde{B}}(b_{i})\right]f\chi_{X\backslash\frac{6}{5}B}\right)\biggr]\\
&- m_{Q}\left[I_{\alpha}
\left(\prod_{i=1}^{k}\left[b_{i}-m_{Q}(b_{i})\right]f\chi_{X\backslash\frac{6}{5}Q}\right)\right] \biggr|\\
\leq&\biggl| m_{Q}\left[I_{\alpha}
\left(\prod_{i=1}^{k}\left[b_{i}-m_{\tilde{B}}(b_{i})\right]f\chi_{X\backslash6^{N}B}\right)\right]\\
&-m_{B}\left[I_{\alpha}
\left(\prod_{i=1}^{k}\left[b_{i}-m_{\tilde{B}}(b_{i})\right]f\chi_{X\backslash6^{N}B}\right)\right]\biggr|\\
+&\biggl|m_{Q}\left[I_{\alpha}
\left(\prod_{i=1}^{k}\left[b_{i}-m_{Q}(b_{i})\right]f\chi_{X\backslash6^{N}B}\right)\right]\\
&-m_{Q}\left[I_{\alpha}
\left(\prod_{i=1}^{k}\left[b_{i}-m_{\tilde{B}}(b_{i})\right]f\chi_{X\backslash6^{N}B}\right)\right]\biggr|\\
+&\left|m_{B}\left[I_{\alpha}
\left(\prod_{i=1}^{k}\left[b_{i}-m_{\tilde{B}}(b_{i})\right]f\chi_{{6^{N}B
\backslash\frac{6}{5}B}}\right)\right]\right|\\
+&\left|m_{Q}\left[I_{\alpha}
\left(\prod_{i=1}^{k}\left[b_{i}-m_{Q}(b_{i})\right]f\chi_{{6^{N}B
\backslash\frac{6}{5}Q}}\right)\right]\right|\\
=:& JJ_{1}+JJ_{2}+JJ_{3}+JJ_{4}.
\end{split}
\end{equation}

With the similar  estimate for $II_{3}$, we easily get that
$$JJ_{1}\leq C||\vec{b}||_{\ast}M^{(\alpha)}_{r,(5)}f(x).$$

To estimate $JJ_{2}$, with the help of (4.7), we deduce that
\begin{equation}
\begin{split} &\biggl|I_{\alpha}
\left(\prod_{i=1}^{k}\left[b_{i}-m_{Q}(b_{i})\right]f\chi_{X\backslash6^{N}B}\right)(y)\\
&-I_{\alpha}
\left(\prod_{i=1}^{k}\biggl[b_{i}-m_{\tilde{B}}(b_{i})\biggr]f\chi_{X\backslash6^{N}B}\right)(y)\biggr|\\
=&\biggl|I_{\alpha}\left(\prod_{i=1}^{k}\left[b_{i}-m_{Q}(b_{i})\right]f\chi_{X\backslash6^{N}B}\right)(y)\\
&\ -\sum_{i=0}^{k}\sum_{\sigma\in
C_{i}^{k}}\left[m_{Q}(b)-m_{\tilde{B}}(b)\right]_{\sigma'}I_{\alpha}\left([b-m_{Q}(b)]_{\sigma}f\chi_{X\backslash6^{N}B}\right)(y)\biggr|\\
\leq &C\sum_{i=0}^{k-1}\sum_{\sigma\in
C_{i}^{k}}(K_{B,Q})^{k-i}||\vec{b}_{\sigma'}||_{\ast}\left|I_{\alpha}\left([b-m_{Q}(b)]_{\sigma}f\chi_{X\backslash6^{N}B}\right)(y)\right|\\
\leq &C\sum_{i=0}^{k-1}\sum_{\sigma\in
C_{i}^{k}}(K_{B,Q})^{k-i}||\vec{b}_{\sigma'}||_{\ast}\biggl\{I_{\alpha}\left([b-m_{Q}(b)]_{\sigma}f\right)(y)\\
&\ \ +I_{\alpha}\left([b-m_{Q}(b)]_{\sigma}f\chi_{6^{N}B}\right)(y)\biggr\}\\
\leq &C\sum_{i=0}^{k-1}\sum_{\sigma\in
C_{i}^{k}}(K_{B,Q})^{k-i}||\vec{b}_{\sigma'}||_{\ast}\biggl\{\sum_{j=0}^{i}\sum_{\eta\in
C_{j}^{i}}|[b(y)-m_{Q}(b)]_{\eta'}|I_{\alpha,\vec{b}_{\eta}}f(y)|\\
&+\biggl|I_{\alpha}\left([b-m_{Q}(b)]_{\sigma}f\chi_{6^{N}B\backslash\frac{6}{5}Q}\right)(y)\biggr|\\
&+\biggl|I_{\alpha}\left([b-m_{Q}(b)]_{\sigma}f\chi_{\frac{6}{5}Q}\right)(y)\biggr|\biggr\}.\\
\end{split}
\end{equation}
Using the H\"{o}lder's inequality and the fact that $Q$ is a
doubling ball, it follows that
\begin{equation}
\begin{split}
&\frac{1}{\mu(Q)}\int_{Q}|[b(y)-m_{Q}(b)]_{\eta'}||I_{\alpha,\vec{b}_{\eta}}f(y)|d\mu(y)\\
\leq
&C\frac{\mu(6Q)}{\mu(Q)}\biggl[\frac{1}{\mu(6Q)}\int_{Q}|[b(y)-m_{Q}(b)]_{\eta'}|^{r'}d\mu(y)\biggr]^{1/r'}\\
&\
\times\biggl[\frac{1}{\mu(6Q)}\int_{Q}|I_{\alpha,\vec{b}_{\eta}}f(y)|^{r}d\mu(y)\biggr]^{1/r}\\
\leq
&C||\vec{b}_{\eta'}||_{\ast}M_{r,(6)}(I_{\alpha,\vec{b}_{\eta}}f)(x).
\end{split}
\end{equation}
By the H\"{o}lder's inequality, Lemma 4.1 and the condition (i) in
Definition 1.6, it is easy to see that for $y\in Q$,
\begin{equation}
\begin{split}
&\left|I_{\alpha}\left([b-m_{Q}(b)]_{\sigma}f\chi_{6^{N}B\backslash\frac{6}{5}Q}\right)(y)\right|\\
\leq
&C\int_{6^{N}B\backslash\frac{6}{5}Q}\frac{1}{\lambda(y,d(y,z))^{1-\alpha/n}}|[b(z)-m_{Q}(b)]_{\sigma}||f(z)|d\mu(z)\\
\leq& C\frac{1}{\mu(5\times 6^{N}B)^{1-\alpha/n}}\int_{6^{N}B}|[b(z)-m_{Q}(b)]_{\sigma}||f(z)|d\mu(z)\\
\leq& C\left[\frac{1}{\mu(5\times
6^{N}B)}\int_{6^{N}B}|[b(z)-m_{Q}(b)]_{\sigma}|^{r'}d\mu(z)\right]^{\frac{1}{r'}}\\
&\ \ \times\left[\frac{1}{\mu(5\times
6^{N}B)^{1-\frac{\alpha r}{n}}}\int_{6^{N}B}|f(z)|^{r}d\mu(z)\right]^{\frac{1}{r}}\\
\leq& C||\vec{b}_{\sigma}||_{\ast}M^{(\alpha)}_{r,(5)}f(x).
\end{split}
\end{equation}
Taking the mean over $y\in Q$, it easily obtains that
\begin{equation}
m_{Q}\left[\left|I_{\alpha}\left([b-m_{Q}(b)]_{\sigma}f\chi_{6^{N}B\backslash\frac{6}{5}Q}\right)\right|\right]
\leq C||\vec{b}_{\sigma}||_{\ast}M^{(\alpha)}_{r,(5)}f(x).
\end{equation}
With the similar method to estimate $II_{1}$, we deduce
\begin{equation}
m_{Q}\left[\left|I_{\alpha}\left([b-m_{Q}(b)]_{\sigma}f\chi_{\frac{6}{5}Q}\right)\right|\right]
\leq C||\vec{b}_{\sigma}||_{\ast}M^{(\alpha)}_{r,(5)}f(x).
\end{equation}
Therefore,
\begin{equation}
\begin{split}
JJ_{2}\leq& C (K_{B,Q})^{k}\biggl[\sum_{i=0}^{k-1}\sum_{\sigma\in
C_{i}^{k}}||\vec{b}_{\sigma}||_{\ast}M_{r,(6)}(I_{\alpha,\vec{b}_{\sigma}}f)(x)
+||\vec{b}||_{\ast}M^{(\alpha)}_{r,(5)}f(x)\biggr]\\
=& C (K_{B,Q})^{k}\biggl[\sum_{i=1}^{k-1}\sum_{\sigma\in
C_{i}^{k}}||\vec{b}_{\sigma}||_{\ast}M_{r,(6)}(I_{\alpha,\vec{b}_{\sigma}}f)(x)\\
&\ \
+||\vec{b}||_{\ast}\biggl(M_{r,(6)(I_{\alpha}f)(x)}+M^{(\alpha)}_{r,(5)}f(x)\biggr)\biggr].
\end{split}
\end{equation}

Now we turn to estimate $JJ_{3}$.  By the H\"{o}lder's inequality,
Lemma 4.1 and Lemma 4.2, for $y\in B$, we get
\begin{equation}
\begin{split}
&\left|I_{\alpha}(\prod^{k}_{i=1}[b_{i}-m_{\tilde{B}}(b_{i})]f\chi_{6^{N}B\backslash\frac{6}{5}B})(y)\right|\\
\leq &
C\sum_{j=1}^{N-1}\frac{1}{\lambda(y,r_{6^{j}B})^{1-\frac{\alpha}{n}}}\int_{6^{j+1}B}
\prod_{i=1}^{k}|b_{i}(z)-m_{\tilde{B}}(b_{i})||f(z)|d\mu(z)\\
&\ +\frac{C}{\lambda(y,r_{B})^{1-\frac{\alpha}{n}}}\int_{6B\backslash\frac{6}{5}B}\prod_{i=1}^{k}|b_{i}(z)-m_{\tilde{B}}(b_{i})||f(z)|d\mu(z)\\
\leq &C \sum^{N-1}_{j=1}\biggl[\frac{\mu(5\times
6^{j+1}B)}{\lambda(y,r_{5\times
6^{j+1}B})}\biggr]^{1-\frac{\alpha}{n}}\biggl[\frac{1}{\mu(5\times
6^{j+1}B)^{1-\frac{\alpha
r}{n}}}\int_{6^{j+1}B}|f(z)|^{r}d\mu(z)\biggr]^{1/r}\\
&\ \times\biggl[\frac{1}{\mu(5\times
6^{j+1}B)}\int_{6^{j+1}B}\prod_{i=1}^{k}|b_{i}(z)-m_{\widetilde{6^{j+1}B}}(b_{i})\\
&\ \ +m_{\widetilde{6^{j+1}B}}(b_{i})-m_{\tilde{B}}(b_{i})|^{r'}d\mu(z)\biggr]^{1/r'}\\
&\
+C\biggl[\frac{1}{\mu(5\times6B)}\int_{6B}\prod_{i=1}^{k}|b_{i}(z)-m_{\tilde{B}}(b_{i})|^{r'}d\mu(z)\biggr]^{1/r'}\\
&\ \times\biggl[\frac{1}{\mu(5\times6B)^{1-\frac{\alpha r}{n}}}\int_{6B}|f(z)|^{r}d\mu(z)\biggr]^{1/r}\\
\leq & CK^{(\alpha)}_{B,Q}(K_{B,Q})^{k}||\vec{b}||_{\ast}M^{(\alpha)}_{r,(5)}f(x)+C||\vec{b}||_{\ast}M^{(\alpha)}_{r,(5)}f(x)\\
\leq &
CK^{(\alpha)}_{B,Q}(K_{B,Q})^{k}||\vec{b}||_{\ast}M^{(\alpha)}_{r,(5)}f(x).
\end{split}
\end{equation}
Taking the mean over $y\in Q$, we obtain
$$JJ_{3}\leq CK^{(\alpha)}_{B,Q}(K_{B,Q})^{k}||\vec{b}||_{\ast}M^{(\alpha)}_{r,(5)}f(x).$$

Finally, we estimate $JJ_{4}$. For $y\in Q$,
\begin{equation}
\begin{split}
&\biggl|I_{\alpha}(\prod_{i=1}^{k}[b_{i}-m_{Q}(b_{i})]f\chi_{6^{N}B\backslash\frac{6}{5}Q})(y)\biggr|\\
\leq &C
\int_{6^{N}B\backslash\frac{6}{5}Q}\frac{1}{\lambda(y,d(y,z))^{1-\frac{\alpha}{n}}}\prod_{i=1}^{k}|b_{i}(z)-m_{Q}(b_{i})||f(z)|d\mu(z)\\
\leq&C \biggl[\frac{1}{\mu(5\times 6^{N}B)}\int_{6^{N}B}\prod_{i=1}^{k}|b_{i}(y)-m_{B}(b_{i})|^{r'}d\mu(y)\biggr]^{1/r'}\\
&\ \ \times\biggl[\frac{1}{\mu(5\times 6^{N}B)^{1-\frac{\alpha r}{n}}}\int_{6^{N}B}|f(y)|^{r}d\mu(y)\biggr]^{1/r}\\
\leq&C||\vec{b}||_{\ast}M^{(\alpha)}_{r,(5)}f(x).
\end{split}
\end{equation}
Taking the mean over $y\in Q$, it follows that
$$JJ_{4}\leq C||\vec{b}||_{\ast}M^{(\alpha)}_{r,(5)}f(x).$$
The estimate for $JJ_{1}$, $JJ_{2}$, $JJ_{3}$ and $JJ_{4}$ yields
(4.6). Thus the proof of Theorem 1.3 is completed.
\end{proof}

%%%%%%%%%%%%%%%%
% bibliography
%%%%%%%%%%%%%%%

% Set bibliography items using the "thebibliography" environment  and following
% the style used by the AMS journals.
%
% If the bibliography is generated by a bibtex database, use "amsplain" or
% "amsalpha" as bibliography style

\end{document}